\documentclass[a4paper,12pt]{amsart}
\usepackage{comment,amsmath,amssymb,amsthm,changepage,pifont,
            graphicx,tikz,soul,xcolor,longtable,mathtools}
\usepackage[new]{old-arrows}
\usepackage[english]{babel}
\usepackage[foot]{amsaddr}
\usetikzlibrary{patterns}
\usetikzlibrary{decorations.pathreplacing}
\usepackage{ytableau}
\allowdisplaybreaks

\newtheorem{theorem}{Theorem}

\newtheorem{lemma}[theorem]{Lemma}

\newtheorem{proposition}[theorem]{Proposition}

\theoremstyle{remark}

\newcommand{\FF}{\mathbf{F}}

\newcommand{\ZZ}{\mathbf{Z}}

\newcommand{\QQ}{\mathbf{Q}}

\newcommand{\CC}{\mathbf{C}}

\newcommand{\OO}{\mathcal{O}}

\newcommand{\cC}{\mathcal{C}}

\DeclareMathOperator{\AGL}{AGL}

\DeclareMathOperator{\Gal}{Gal}
\DeclareMathOperator{\mult}{mult}
\DeclareMathOperator{\Ind}{Ind}
\DeclareMathOperator{\Res}{Res}
\DeclareMathOperator{\Hom}{Hom}

\DeclareMathOperator{\GL}{GL}

\DeclareMathOperator{\PSL}{PSL}

\usepackage[colorinlistoftodos]{todonotes}
\usepackage[colorlinks=true, allcolors=blue]{hyperref}

\title{Arithmetically equivalent number fields have approximately the same successive minima}

\author{Floris Vermeulen}
\address{Department of Mathematics, KU Leuven, Belgium}
\email{floris.vermeulen@kuleuven.be}

\date{}

\begin{document}

\begin{abstract}
Let $K$ and $K'$ be arithmetically equivalent number fields, both of degree $d$. We prove that $K$ and $K'$ have the same successive minima, up to a constant depending only on $d$. We give examples showing that one cannot expect equality.
\end{abstract}

\maketitle

\section{Introduction}\label{sec:introduction}

Two number fields $K$ and $K'$ are said to be \emph{arithmetically equivalent} if their Dedekind zeta functions are equal: $\zeta_K(s) = \zeta_{K'}(s)$. Arithmetically equivalent number fields are not necessarily isomorphic, but they always have the same degree, discriminant, signature, Galois closure and more, by a result of Perlis~\cite{perlis1} using Gassmann's theorem~\cite{gassmann}. However, for instance their class number and regulator can differ~\cite{SmitPerlis}. One of the main problems is to figure out how ``isomorphic" arithmetically equivalent fields are, i.e.\ which arithmetic invariants are equal and which ones can differ. We refer to~\cite{sutherland, sutherland_strong, solomatin} for a more elaborate discussion. 

In this article, we study the successive minima of arithmetically equivalent number fields, which we briefly recall. Let $K$ be a degree $d$ number field with ring of integers $\OO_K$ and denote by $\sigma_1, \ldots, \sigma_d: K\to \CC$ the embeddings of $K$ into $\CC$. Then the \emph{Minkowski embedding} is the map
\[
\iota: K\to \CC^d: v\mapsto (\sigma_1(v), \ldots, \sigma_d(v)).
\]
This embedding yields a norm $|| \cdot ||$ on $K$, which for $v\in K$ we normalize as
\[
||v|| = \sqrt{\frac{1}{d}\big(|\sigma_1(v)|^2 + \ldots + |\sigma_d(v)|^2\big)}.
\]
There are different conventions in the literature regarding the Minkowski embedding and the corresponding norm on $K$, for instance by treating the real and complex embeddings separately. However, the above choice of embedding is the most natural in many settings, including ours (see e.g.\ \cite{peikertrosen, neukirch}). Then the $i$-th \emph{successive minimum $\lambda_i$ of $K$} is the smallest real number such that the set
\[
\{v\in \OO_K \mid ||v|| \leq \lambda_i\}
\]
contains $i$ $\QQ$-linearly independent elements. These successive minima can be seen as an approximate refinement of the discriminant of $K$. Indeed, by Minkowski's second theorem~\cite[Thm.\,16]{siegel}, the successive minima satisfy
\[
\lambda_1\lambda_2\cdots \lambda_d = \Theta_d\left(\sqrt{|\Delta_K|}\right),
\]
where $\Delta_K$ is the discriminant of $K$. Since arithmetically equivalent number fields have the same discriminant, it is natural to wonder whether they have (approximately) the same successive minima as well. This problem seems not to have been studied before, and we give a positive answer to this question.

\begin{theorem}\label{thm:main.thm}
Let $K$ and $K'$ be two arithmetically equivalent number fields of degree $d$. Let $\lambda_1 \leq \ldots\leq \lambda_d$ and $\lambda_1'\leq \ldots\leq \lambda_d'$ be the multiset of successive minima of $K$ and $K'$. Then there is a constant $c_d$ depending only on $d$ such that for all $i$
\[
\lambda_i \leq c_d \lambda_i'.
\]
\end{theorem}

The strength of this result lies in the independence of the constant $c_d$ on the specific number fields $K$ and $K'$. It is natural to normalize the successive minima and instead look at the quantities
\[
\mu_1 = \frac{\log \lambda_1}{\log(|\Delta_K|)/2}, \ldots, \mu_d = \frac{\log \lambda_d}{\log(|\Delta_K|)/2}.
\]
Indeed, Minkowski's second theorem implies that the $\mu_i$ sum up to roughly 1, i.e.\ $\sum_i \mu_i = 1 + O_d(\log(|\Delta_K|)^{-1})$. Now we can rephrase our result as stating that
\[
|\mu_i - \mu_i'| \leq \frac{\log c_d}{\log |\Delta_K|}.
\]
So if the absolute discriminant $|\Delta_K|$ is large compared to $c_d$, then the $\mu_i$ and $\mu_i'$ are very close together. Of course, the strength of the result also depends on the size of the constant $c_d$. In Section~\ref{sec:betterbounds} we will prove that one may take $c_d = d$, and additionally provide a more direct comparison between the lattices $\OO_K$ and $\OO_{K'}$ by constructing an injective linear map $\OO_K\to \OO_{K'}$ of operator norm bounded by $d^2$.

We give two proofs of Theorem~\ref{thm:main.thm}. The first one passes through the common Galois closure $L$ of $K$ and $K'$ and expresses the successive minima of $K$ and $K'$ in terms of the successive minima of lattices in $L$ coming from irreducible representations of $\Gal(L/\QQ)$. There is an analogue of Theorem~\ref{thm:main.thm} in the function field setting, where one has an exact equality of scrollar invariants of Gassmann equivalent function fields~\cite{syzrep} (see~\cite[Sec.\,7]{hessRR} for this analogy), and this proof is similar in spirit. However, compared to the function field setting this proof does not seem to give good bounds on the required constant $c_d$. In particular, a rough argument shows that this proof gives $c_d = (d!)^3$. 

The second proof is more direct, by explicitly constructing a $\ZZ$-linear map $\OO_K\to \OO_K'$ using ideas from~\cite{bosmasmit_classnumberrelations, sutherland}. Here we obtain much better bounds on $c_d$, in particular proving Theorem~\ref{thm:main.thm} with $c_d = d$. Moreover, this method is also well-suited to explicit computations in low degree. We prove in Section~\ref{sec:example} that one may take $c_7 = 3$ and $c_8 = 3$.

\subsection*{Acknowledgements} The author thanks Wouter Castryck for proposing this question, helpful discussions, and for providing the code in section \ref{sec:example}. The author thanks Aurel Page and Fabrice Etienne for interesting discussions and for providing the argument to get $c_d = d$. The author thanks Casper Putz for interesting discussions. The author is supported by F.W.O. Flanders (Belgium) with grant number 11F1921N.

\section{Some representation theory}

In this section, we recall two representation theoretic lemmas which we will need. Throughout, $G$ will be a finite group. For a representation $V$, we will denote its character by $\chi_V$. For two representations $U, V$ of $G$, we denote by $\langle V, U\rangle$ the inner product $\langle \chi_V, \chi_U\rangle$. 

\begin{lemma}\label{lem:projection.isotypic}
Let $V$ be an irreducible $\QQ$-representation of $G$. Then there exists an $f_V \in \ZZ[G]$ such that $f_V$ induces the zero map on $U$ for every irreducible $\QQ$-representation $U\neq V$ of $G$ and such that $f_V$ is invertible as a $\QQ$-linear map on $V$.
\end{lemma}

\begin{proof}
Consider the element
\[
f_V = \sum_{g\in G} \chi_V (g^{-1}) g \in \ZZ[G]
\]
and note that it induces the zero map on every irreducible $\QQ$-representation $U\neq V$ of $G$, see e.g.\ ~\cite[Thm.\,8]{serreRepr}. On the other hand, we see that on $V\otimes \CC$ the map $f_V$ becomes invertible, and hence the same holds over $\QQ$.
\end{proof}

\begin{lemma}\label{lem:dim.fixed}
Let $V$ be a $\QQ$-representation of $G$ and let $H$ be a subgroup of $G$. Then the dimension of $V^H$ is equal to $\langle V, \QQ[G/H] \rangle$.
\end{lemma}

\begin{proof}
Denote by $1$ the trivial representation. We have that
\begin{align*}
\dim V^H &= \mult(1, \Res_{H}^G V) = \dim_\QQ \Hom_H(1, \Res_H^G V) \\
	&= \dim_\QQ \Hom_G(\Ind_H^G 1, V) = \langle V, \QQ[G/H] \rangle,
\end{align*}
where we have used Frobenius reciprocity.
\end{proof}

Note that if $V$ is absolutely irreducible, i.e.\ $\langle V, V \rangle = 1$, then $\langle V, \QQ[G/H]\rangle$ is equal to the multiplicity of $V$ in $\QQ[G/H]$.

\section{Successive minima of arithmetically equivalent number fields}\label{sec:main.thm}

In this section we give a first proof of Theorem \ref{thm:main.thm}. We first introduce some useful notation.

Let $A$ be some object, e.g.\ an integer or a group. For two multisets of positive real numbers $\{\lambda_1, \ldots, \lambda_m\}$ and $\{\lambda_1', \ldots, \lambda_m'\}$ we define 
\[
\{\lambda_1, \ldots, \lambda_m\} \ll_A  \{\lambda_1', \ldots, \lambda_m'\}
\]
if there exists a constant $c=c_A>0$ depending on $A$ such that after reordering both multisets as $\lambda_1\leq \ldots \leq \lambda_m, \lambda_1'\leq \ldots \leq \lambda_m'$, we have for all $i$ that
\[
\lambda_i \leq c \lambda_i'.
\]
We use the notation $\{\lambda_i\}_i =_A \{\lambda_i'\}_i$ to mean that $\{\lambda_i\}_i \ll_A \{\lambda_i'\}_i$ and $\{\lambda_i'\}_i\ll_A \{\lambda_i\}_i$. For a multiset $\{\lambda_1, \ldots, \lambda_m\}$ and an integer $k$ we define $k\cdot\{\lambda_1, \ldots, \lambda_m\}$ to be the multiset $\{\lambda_1, \ldots, \lambda_1, \lambda_2, \ldots, \lambda_2, \ldots, \lambda_m, \ldots, \lambda_m\}$ where every $\lambda_i$ appears $k$ times.

Let $K$ be a degree $d$ number field with ring of integers $\OO_K$. We follow the notation from Section~\ref{sec:introduction} regarding the Minkowski embedding, the resulting norm and the successive minima. Note that if $L/K$ is a field extension then the resulting norm on $L$ extends the one on $K$, precisely because of the normalization constant $1/\sqrt{d}$ in the definition of the norm. It is not necessarily true that there exists a basis of $\OO_K$ as a $\ZZ$-module achieving the successive minima, but we will not need more technical notions such as Minkowski reduced bases. For more background on successive minima we refer to \cite{siegel}.

Let $L$ be a Galois extension of $\QQ$ with Galois group $G$. By the normal basis theorem, $L$ is isomorphic to $\QQ[G]$ as a $\QQ[G]$-module. For $V$ an irreducible $\QQ$-representation of $G$, we denote by $L_V$ the \emph{isotypic component} in $L$ corresponding to $V$. Recall that this consists of all elements of $L$ which generate $V$ as a representation (together with $0$). Since we are in characteristic $0$, $\QQ[G]$ is semi-simple and is the direct sum of its isotypic components. Define $\OO_V = \OO_L\cap L_V$, which is a sublattice of $\OO_L$ of rank $\dim(V)^2 / \langle V, V\rangle$. The successive minima of $\OO_V$ will play a central role in the proof of Theorem~\ref{thm:main.thm}, and the following lemma may be of independent interest.

\begin{lemma}\label{lem:successive.minima.isotypical}
Let $L$ be a Galois extension of $\QQ$ with Galois group $G$, and let $V$ be an irreducible $\QQ$-representation of $G$. Then the multiset of successive minima of $\OO_V$ is of the form
\[
\dim(V)\cdot \{\lambda_{V, 1}, \lambda_{V, 2}, \ldots, \lambda_{V, n_V} \},
\]
where $n_V = \dim(V) / \langle V, V \rangle$. 
\end{lemma}

In other words, this lemma states that every successive minimum of $\OO_V$ appears with multiplicity $\dim (V)$. We call the $\lambda_{V, i}$ the \emph{succesive minima of $V$ in $L$}.

\begin{proof}
The proof proceeds inductively. The main point is that if $v$ is in $L$ and $g$ in $G$, then $||v|| = ||gv||$. Take a non-zero vector $v_1$ in $\OO_V$ for which $||v_1||$ is minimal. By acting with $G$ on $v_1$, we obtain $\dim(V)$ linearly independent vectors in $\OO_V$ with the same length, showing that the first successive minimum appears with multiplicity $\dim (V)$. Now take a vector $v_2$ of shortest length which is linearly independent from $\{gv_1\mid g\in G\}$, and act with $G$ on $v_2$. Continuing in this manner gives the result.
\end{proof}

Our main tool is the following lemma.

\begin{lemma}\label{lem:successive.minima.subfield}
Let $L$ be a Galois extension of $\QQ$ with Galois group $G$ and let $H$ be a subgroup of $G$. Let $\{\lambda_1, \ldots, \lambda_m\}$ be the successive minima of $L^H$. Then 
\[
\{\lambda_1, \ldots, \lambda_m\} =_G \bigcup_{V} \langle V, \QQ[G/H] \rangle \cdot \{\lambda_{V, 1}, \ldots, \lambda_{V, n_V}\},
\]
where the union is as a multiset, and is over all irreducible $\QQ$-representations $V$ of $G$.
\end{lemma}

\begin{proof}
Recall that the norm on $L$ extends the norm on $L^H$. Hence the successive minima of $L^H$ are the same as the successive minima of the lattice $\OO_{L^H}$ considered as a sublattice of $\OO_L$. So it is enough to prove the result when the $\lambda_i$ are the successive minima of $\OO_{L^H}$ inside $L$.

Let $v_1, \ldots, v_m$ be $\QQ$-linearly independent elements of $\OO_{L^H}$ such that $||v_i|| = \lambda_i$ under the Minkowski embedding of $L$. This is not necessarily a $\ZZ$-basis of $\OO_{L^H}$ but it will suffice for our purpose. For every irreducible $\QQ$-representation $V$, let $w_{V, 1}, \ldots, w_{V, n_V}$ be in $\OO_V$ such that acting with $G$ on $\{w_{V, i}\}_i$ gives a $\QQ$-basis for the isotypic component $L_V$, and such that $||w_{V, i}|| = \lambda_{V, i}$ are the successive minima of $V$ in $L$. We then prove that
\[
\{||v_1||, \ldots, ||v_m||\} =_G \bigcup_{V} \langle V, \QQ[G/H]\rangle \cdot \{||w_{V, 1}||, \ldots, ||w_{V, n_V}||\}.
\]
For one inequality, consider the $\ZZ$-lattice in $\OO_{L^H}$ generated by all elements of the form
\[
w_{V, i, g} = \sum_{h\in H} hgw_{V, i}, 
\]
where $g\in G$, $V$ is an irreducible $\QQ$-representation of $G$ and $i=1, \ldots, n_V$. The result is a full rank sublattice of $\OO_{L^H}$. Indeed, if $v$ is in $L^H$ then we can write 
\[
v = \sum_{V, i, g} a_{V, i, g} gw_{V, i},
\]
where the sum is over all irreducible $\QQ$-representations $V$, $i=1, \ldots, n_V$, $g\in G$ and the $a_{V, i, g}$ are in $\QQ$. Then acting with $\sum_{h\in H} h$ gives that
\[
v = \frac{1}{\# H} \sum_{V, i, g} a_{V, i, g} w_{V, i, g}.
\]
Now, since $||w_{V, i, g}|| =_G ||w_{V, i}||$ Lemma \ref{lem:dim.fixed} implies that
\[
\{||v_1||, \ldots, ||v_m||\} \ll_G \bigcup_{V} \langle V, \QQ[G/H]\rangle \cdot \{||w_{V, 1}||, \ldots, ||w_{V, n_V}||\}.
\]

For the other direction we follow an inductive procedure. Using Lemma \ref{lem:projection.isotypic} we fix for every irreducible representation $V$ of $G$ an element $f_V\in \ZZ[G]$ which is invertible as a $\QQ$-linear map on $V$ and acts as the zero map on $U$, for every irreducible $\QQ$-representation $U\neq V$. Note that every $f_V$ maps elements of $\OO_L$ to $\OO_V$. For every $v\in L$ and every $V$ we have that
\[
||f_V(v)|| \ll_G ||v||.
\]
We now inductively pick an irreducible representation $V_i$ of $G$ for $i=1, \ldots, m$ in the following way. We start by taking an irreducible $\QQ$-representation $V_1$ of $G$ such that $f_{V_1}(v_1)$ is non-zero. If $V_1, \ldots, V_i$ have already been constructed, we take any $V_{i+1}$ such that $f_{V_{i+1}}(v_{i+1})$ is $\QQ$-linearly independent from 
\[
\{f_{V_{i+1}}(v_j) \mid j=1, \ldots, i, \text{ such that } V_j = V_{i+1}\}.
\]
Note that we can indeed always take such a $V_{i+1}$ since $\sum_V f_V$ is an invertible $\QQ$-linear map on $L$. Having picked all $V_i$, we will have that for each irreducible $\QQ$-representation $V$ the set
\[
\{ f_V(v_i) \mid V_i = V\}
\]
will generate a sublattice of $\OO_V$ whose rank is $n_V \langle V, \QQ[G/H]\rangle $. Even more, for a fixed $i$ the intersection of this lattice with the $\QQ$-span of all conjugates of $w_{V, i}$ has rank $\langle V, \QQ[G/H]\rangle$. Indeed, this follows from Lemma~\ref{lem:dim.fixed} and the fact that $f_V$ is invertible on $L_V$. Because $||f_V(v)|| \ll_G ||v||$, we therefore have that
\[
\langle V, \QQ[G/H]\rangle \cdot \{||w_{V, 1}||, \ldots, ||w_{V, n_V}||\} \ll_G \{||v_i|| \mid i=1, \ldots, m \text{ for which } V_i = V\}. 
\]
Taking the union over all $V$ finishes the proof.
\end{proof}

With this, we can prove that arithmetically equivalent number fields have almost the same successive minima.

\begin{proof}[Proof of Theorem \ref{thm:main.thm}]
We have two arithmetically equivalent number fields $K$ and $K'$ of degree $d$. By Gassmann's theorem~\cite{gassmann} $K$ and $K'$ have a common Galois closure $L$ with Galois group $G$ and the subgroups $H, H'$ of $G$ for which $K = L^H, K' = L^{H'}$ are \emph{Gassmann equivalent}. This means that for every conjugacy class $\cC$ of $G$ we have $\#(\cC\cap H) = \#(\cC\cap H')$. Equivalently, this means that the $\QQ[G]$-modules $\QQ[G/H]$ and $\QQ[G/H']$ are isomorphic, see e.g.\ \cite[Lem.\,2.7, Thm.\,2.8]{sutherland}. But then Lemma \ref{lem:successive.minima.subfield} implies that the multisets of successive minima of $K$ and $K'$ are asymptotic to each other, up to a factor depending on $G$. The dependence on $G$ can be replaced by a dependence on $d$, since $G$ is isomorphic to a subgroup of $S_d$. 
\end{proof}

A rough analysis shows that the above proof gives the constant $c_d = (d!)^3$. However, in the next section we will give an alternative proof of Theorem~\ref{thm:main.thm} which yields that we may take $c_d = d$.


\section{Better bounds}\label{sec:betterbounds}

In this section we present an alternative proof of Theorem~\ref{thm:main.thm}, which will in particular allow us to prove the result with $c_d=d$. In fact, the same proof allows us to relate the shape of the lattices of arithmetically equivalent number fields more directly as follows.

\begin{proposition}\label{prop:operator.norm}
Let $K, K'$ be arithmetically equivalent number fields of degree $d$. Then there exists a linear map $\phi: K\to K'$ such that
\begin{enumerate}
\item $\phi$ is an isomorphism of $\QQ$-vector spaces,
\item $\phi(\OO_K) \subset \OO_{K'}$, and
\item for every $v\in \OO_K$ we have $||\phi(v)|| \leq d^2||v||$.
\end{enumerate}
In other words, $\phi$ has operator norm bounded by $d^2$.
\end{proposition}

In Section~\ref{sec:example}, we analyse this proof of this proposition further to improve this operator norm for degree $7$ and $8$.

Since the proof of this proposition and of Theorem~\ref{thm:main.thm} with $c_d=d$ is similar, we give both at the same time.

\begin{proof}[Proof of Theorem~\ref{thm:main.thm} and of Proposition~\ref{prop:operator.norm}]
Let $K$ and $K'$ be arithmetically equivalent number fields of degree $d$. Let $L$ be their common Galois closure with Galois group $G$ over $\QQ$. Let $H = \Gal(L/K)$ and $H' = \Gal(L/K')$. We will show how to construct elements $\phi\in \ZZ[G]$ inducing a morphism $\phi: K\to K'$ with the desired properties.

To construct $\phi$, let $g_1, \ldots, g_d$ be a complete set of representatives for the left cosets $G/H$. We take $\phi$ of the form
\[
\phi = \sum_i c_{g_iH} g_i,
\]
where $c_{g_iH}\in \ZZ$. Then it is already clear that $\phi(\OO_K)\subset \OO_K'$. For $\phi$ to map $K$ to $K' = L^{H'}$, we need the following for every $h'\in H'$. For every $i$ there is a $\pi(i)$ for which $h'g_i = g_{\pi(i)} h_i$ for some $h_i\in H$. Then we must have for every $v\in K$ that
\[
h'\phi(v) = \sum_i c_{g_iH} h'g_i v = \sum_i c_{g_iH} g_{\pi(i)} v = \sum_i c_{h'g_iH} g_{i}v.
\]
Therefore, we must have that $c_{h'g_iH} = c_{g_iH}$ for every $h'\in H'$ and every $i$. In other words, $c$ is constant on the double cosets $H'\backslash G/H$. By~\cite[Lem.\,4.5]{sutherland} there is a bijective correspondence between $\ZZ$-valued functions on double cosets $H'\backslash G/H$ and $\ZZ[G]$-morphisms $\ZZ[G/H]\to \ZZ[G/H']$. In more detail, if we let $g_1', \ldots, g_d'$ be a complete set of representatives for $G/H'$, then a $\ZZ[G]$-morphism $\psi: \ZZ[G/H]\to \ZZ[G/H']$ determines a $d\times d$ matrix $M$ over $\ZZ$ via
\[
\psi(g_iH) = \sum_j M_{ij} g_j' H'.
\]
Then the function $c$ on double cosets corresponding to $\psi$ is given by $H'g_j'^{-1}g_iH\mapsto M_{ij}$. 

We now describe the universal morphism $\ZZ[G/H]\to \ZZ[G/H']$ following~\cite[Sec.\,4]{bosmasmit_classnumberrelations}. For this, let $t = \#(H'\backslash G/H)$ be the number of double cosets, and let $a_1, \ldots, a_t$ be formal variables corresponding to these double cosets. Take $x_1, \ldots, x_t\in G$ a set of representatives for these double cosets. Define the matrix $A$ as $A_{ij} = a_\ell$ if $H'g_j'^{-1}g_iH = H'x_\ell H$. Upon fixing integer values for the $a_i$, we obtain a $\ZZ[G]$-morphism $\ZZ[G/H]\to \ZZ[G/H']$, and conversely, every morphism is obtained from such a matrix. 

Let us first focus on the proof of Proposition~\ref{prop:operator.norm}. Let $f$ be the determinant of the matrix $A$, which is an integer polynomial in the $a_i$ homogeneous of degree $d$. Since $\QQ[G/H]$ and $\QQ[G/H']$ are isomorphic, there exists an injection $\ZZ[G/H]\hookrightarrow \ZZ[G/H']$. For this injection, the corresponding determinant of $A$ is non-zero, and so $f$ is not the zero polynomial. By the Schwartz--Zippel Lemma (see e.g.\ ~\cite[Cor.\,4.1.2]{castryck_dimension_2019}), there exist $(a_i)_i$ for which $f(a_i)\neq 0$ and for which $|a_i|\leq d$ for all $i$. We fix these $a_i$ and let $\psi$ be the corresponding morphism $\ZZ[G/H]\to \ZZ[G/H']$. Let $\phi$ be the map $K\to K'$ as obtained above. By~\cite[Cor.\,4.7]{sutherland}, the invertibility of $\psi$ as a morphism $\QQ[G/H]\to \QQ[G/H']$ guarantees that the map $\phi$ is also invertible as a map of vector spaces. Moreover, we have that $||\phi(v)||\leq \sum_i |A_{ij}| ||v||$ (for any $j$). Hence by our choice of $(a_i)_i$, we have that $||\phi(v)||\leq d^2 ||v||$. 

For the improved version of Theorem~\ref{thm:main.thm}\footnote{The author thanks Aurel Page for explaining this argument.}, we instead construct multiple maps $K\to K'$. For each $i=1, \ldots, t$, let $\phi_i: K\to K'$ be the morphism corresponding to the choice $a_i = 1$ and $a_j = 0$ if $j \neq i$. Note that then $\phi_i$ maps $\OO_K$ to $\OO_{K'}$ and for every $v\in \OO_K$ we have $||\phi_i(v)||\leq d||v||$. None of these maps has to be an isomorphism of vector spaces, but since the construction of $\phi$ is linear in $a$, the previous paragraph implies that some linear combination of the $\phi_i$ is invertible. So if $V\subset K$ is a subspace of dimension $i$, then the spaces $\phi_1(V), \ldots, \phi_t(V)$ together span a vector space of dimension at least $i$. Now if $v_1, \ldots, v_d$ are elements of $\OO_K$ achieving the successive minima of $K$, then applying this with $V = \{v_1, \ldots, v_i\}$ for $i=1, \ldots, d$ proves Theorem~\ref{thm:main.thm} with $c_d=d$.
\end{proof}

A more careful analysis of the determinant $f$ of the matrix $A$ appearing in the previous proof might allow one to prove Proposition~\ref{prop:operator.norm} with the operator norm of $\phi$ bounded by $d$. In the next section, we take a closer look at $f$ in low degree.

\section{Low degree examples}\label{sec:example}

The author thanks Wouter Castryck for providing the code for computing successive minima of number fields in Magma~\cite{magma}.

Throughout this section, let $K$ and $K'$ be two non-isomorphic arithmetically equivalent number fields of degree $d$. This implies that $d\geq 7$. We let $L$ be their common Galois closure with Galois group $G$, and let $H, H'$ be the subgroups of $G$ corresponding to $K, K'$. Let $t$ be the number of double cosets $\#(H'\backslash G/H)$ and let $a_1, \ldots, a_t$ be variables. We denote by $A$ the universal $d\times d$ matrix with entries from the $a_i$ as constructed in the proof from the previous section. We let $\phi: K\to K'$ be the morphism from Proposition~\ref{prop:operator.norm}. Let us analyse the proof of Theorem~\ref{thm:main.thm} and Proposition~\ref{prop:operator.norm} from Section~\ref{sec:betterbounds} in more detail.

\subsection*{Degree 7}

In this case, $G = \GL_3(\FF_2) = \PSL_2(\FF_7)$ is a group of order $168$ and $H$ and $H'$ are both of index $7$~\cite{perlis1}, see also~\cite{BosmaSmit}. We have shown that one may take $c_7 = 7$ and $\phi$ of operator norm $49$ in the previous section. In fact, one may even take $||\phi||\leq 3$ and hence also $c_7 = 3$, as we now argue. There are two double cosets in $H'\backslash G/H$ and the corresponding matrix $A$ in $a = a_1, b = a_2$ is given by
\[
A = \begin{pmatrix}
a & b & b & b & a & a & a \\
b & a & a & b & b & a & a \\
b & b & a & a & a & b & a \\
b & a & b & a & a & a & b \\
a & b & a & a & b & a & b \\
a & a & b & a & b & b & a \\
a & a & a & b & a & b & b
\end{pmatrix}.
\]
The determinant of this matrix is $f=-32(a-b)^6(32a+24b)$. Hence we can take $a=0, b=1$ to get $f\neq 0$. The corresponding morphism $\phi: K\to K'$ will satisfy $||\phi(v)||\leq 3||v||$.

In general however, one cannot expect that $c_7 = 1$. For an explicit example, by \cite{perlis1} we may take $K$ and $K'$ by adjoining a root to $\QQ$ of the respective polynomials
\[
f(x) = x^7-7x+3, \text{ and }  g(x) = x^7+14x^4-42x^2-21x+9.
\]
Then $K$ and $K'$ are arithmetically equivalent but not isomorphic, and their common Galois closure $L$ has Galois group $\GL_3(\FF_2)$. Using Magma, we have computed the successive minima of $K$ and $K'$ to be approximately
\begin{align*}
K: \qquad &\{1.00, 1.29, 1.78, 2.48, 2.57, 2.62, 2.96\}, \\
K': \qquad &\{1.00, 1.82, 1.83, 2.23, 2.31, 2.52, 2.82\}.
\end{align*}
More generally, one can construct an infinite family of arithmetically equivalent number fields whose Galois closures have Galois group $\GL_3(\FF_2)$, of arbitrarily large discriminant. Explicitly, for rational numbers $s,t$ define the polynomial
\begin{align*}
f_{s,t}(x)& = x^7 + (-6t+2)x^6 + (8t^2+4t-3)x^5 + (-s-14t^2+6t-2)x^4 \\
& +(s+6t^2-8t^3-4t+2)x^3 + (8t^3 + 16t^2)x^2 + (8t^3-12t^2)x - 8t^3 \in \QQ[x].
\end{align*}
By work of LaMacchia~\cite{lamacchia}, this polynomial has Galois group $\GL_3(\FF_2)$ over $\QQ(s,t)$. Hence by Hilbert irreducibility there are infinitely many specializations $s,t\in \QQ$ for which $f_{s,t}(x)$ is irreducible with Galois group $\GL_3(\FF_2)$. Bosma and de Smit~\cite{BosmaSmit} proved that if $f_{s,t}(x)$ is irreducible with Galois group $\GL_3(\FF_2)$ then the degree $7$ number fields $K$ and $K'$ defined by $f_{s,t}$ and $f_{-s,t}$ are non-isomorphic and arithmetically equivalent. Experiments in Magma for various values of $s,t\in \QQ$ show that the successive minima of $K$ and $K'$ are very close together. In fact, for these computations the value of $c$, i.e.\ the ratio between the successive minima of $K$ and $K'$, never exceeded $3/2$.

\subsection*{Degree 8}

In degree $8$, by~\cite[Thm.\,3]{BosmaSmit} there are two options for the Galois group $G$ and the corresponding subgroups $H$, $H$. Namely, either $G = \AGL_1(\ZZ/8\ZZ)$, or $G = \GL_2(\FF_3)$.

Assume first that that $G=\AGL_1(\ZZ/8\ZZ)$, which is a group of order $32$. This group is isomorphic to the subgroup of $S_8$ generated by $(12345678), (13)(57)(26)$ and $(15)(37)$. Under this isomorphism, the index $8$ subgroups $H$ and $H'$ which are non-conjugate but Gassmann equivalent are $H=\langle (13)(57)(26), (15)(37)\rangle$ and $H' = \langle (17)(26)(35), (17)(35)(48)\rangle$. There are $4$ double cosets of $H, H'$ in $G$, and the corresponding matrix $A$ in the variables $a,b,c,d$ is
\[
A = \begin{pmatrix}
a & c & d & b & b & d & a & c \\
c & a & c & d & a & b & b & d \\
a & b & d & c & c & d & a & b \\
d & c & a & c & b & a & d & b \\
c & d & b & a & d & c & b & a \\
b & a & b & d & a & c & c & d \\
b & d & c & a & d & b & c & a \\
d & b & a & b & c & a & d & c
\end{pmatrix}.
\]
This matrix has determinant $f = 64(b-c)^4(a-d)^2(a-b-c+d)(a+b+c+d)$. The points which give the best estimates for the constant are $(a,b,c,d) = \pm (2,1,0,0)$ or $\pm (1,1,0,-1)$. Indeed, this gives that $||\phi||\leq 6$, and so the successive minima of $K$ and $K'$ differ by at most a factor $6$. In fact, we can do slightly better by noting that each of the maps $\phi_i$ constructed in Section~\ref{sec:betterbounds} has norm at most $2$. Hence the successive minima of $K$ and $K'$ differ by at most a factor $2$. We note that the morphism corresponding to $(2,1,0,0)$ was also considered in~\cite[p.\ 145]{desmit_brauerkuroda}, where it was used to relate the class numbers of certain families of arithmetically equivalent number fields.

Now assume that $G = \GL_2(\FF_3)$, which is a group of order $48$. This group acts on $\FF_3^2$, and we let $H$ be the stabilizer of $(1,0)$ under this action. Define $H'$ to be $\{g^T\mid g\in H\}$. These subgroups are Gassmann equivalent but non-conjugate. There are $3$ double cosets, and the matrix $A$ in the variables $a,b,c$ is
\[
A = \begin{pmatrix}
a & b & b & c & a & c & a & c \\
c & a & c & b & b & c & a & a \\
c & c & a & c & a & b & b & a \\
b & c & a & a & c & c & a & b \\
b & a & c & c & a & a & c & b \\
c & b & b & a & c & a & c & a \\
a & c & a & b & b & a & c & c \\
a & a & c & a & c & b & b & c
\end{pmatrix}.
\]
The determinant of this matrix is $f = -9(a-c)^4(a-2b+c)^3(a+\frac{2}{3}b+c)$. Taking $a=1, b=c=0$ yields $||\phi||\leq 3$ and shows that the successive minima of $K$ and $K'$ differ by at most a factor $3$.

In conclusion, we may take $c_8 = 3$. Similar techniques can be used to provide bounds in higher degrees as well. 

To end, let us note that in the above examples the determinant $f$ is of a very specific form. Indeed, it always splits as a product of only a few linear factors. A more careful analysis of this determinant might allow one to prove Proposition~\ref{prop:operator.norm} with a bound on $||\phi||$ which is linear in $d$.



\bibliographystyle{amsplain}
\bibliography{MyLibrary}

\providecommand{\bysame}{\leavevmode\hbox to3em{\hrulefill}\thinspace}
\providecommand{\MR}{\relax\ifhmode\unskip\space\fi MR }
\providecommand{\MRhref}[2]{%
  \href{http://www.ams.org/mathscinet-getitem?mr=#1}{#2}
}
\providecommand{\href}[2]{#2}
\begin{thebibliography}{10}

\bibitem{magma}
W.~Bosma, J.~Cannon, and C.~Playoust, \emph{The {M}agma algebra system. {I}.
  {T}he user lanuage}, Journal of {S}ymbolic {C}omputation \textbf{24} (1997),
  no.~3--4, 235--265.

\bibitem{bosmasmit_classnumberrelations}
W.~Bosma and B.~de~Smit, \emph{Class number relations from a computational
  point of view}, J. Symbolic Computation \textbf{31} (2001), 97--112.

\bibitem{BosmaSmit}
\bysame, \emph{On arithmetically equivalent number fields of small degree},
  ANTS proceedings. Lectures Notes in Computer Science \textbf{2369} (2002),
  67--79.

\bibitem{castryck_dimension_2019}
W.~Castryck, R.~Cluckers, P.~Dittmann, and K.~Huu~Nguyen, \emph{The dimension
  growth conjecture, polynomial in the degree and without logarithmic factors},
  Algebra \& Number Theory \textbf{14} (2020), no.~8.

\bibitem{syzrep}
W.~Castryck, F.~Vermeulen, and Y.~Zhao, \emph{Scrollar invariants, syzygies and
  representations of the symmetric group}, arXiv preprint (2022),
  \url{https://arxiv.org/abs/2201.06322}.

\bibitem{desmit_brauerkuroda}
B.~de~Smit, \emph{Brauer–{K}uroda relations for {$S$}-class numbers}, Acta
  Arithmetica \textbf{98} (2001), 133--146.

\bibitem{SmitPerlis}
B.~de~Smit and R.~Perlis, \emph{Zeta functions do not determine class numbers},
  Bulletin of the American Mathematical Society \textbf{31} (1994), no.~2,
  213--215.

\bibitem{gassmann}
F.~Gassmann, \emph{Bemerkungen zur vorstehenden arbeit von {H}urwitz}, Math. Z
  \textbf{25} (1926), 665--675.

\bibitem{hessRR}
F.~Hess, \emph{Computing {Riemann}-{Roch} spaces in algebraic function fields
  and related topics}, Journal of {S}ymbolic {C}omputation \textbf{33} (2002),
  no.~4, 425--445.

\bibitem{lamacchia}
S.E. LaMacchia, \emph{{Polynomials with Galois group PSL(2, 7)}}, Commutative
  Algebra \textbf{8} (1980), 983--992.

\bibitem{neukirch}
J.~Neukirch, \emph{Algebraic number theory}, Springer-{V}erlag {B}erlin, 1999.

\bibitem{peikertrosen}
C.~Peikert and A.~Rosen, \emph{Lattices that admit logarithmic worst-case to
  average-case connection factors}, full version, available at
  \url{https://eprint.iacr.org/2006/444}, of eponymous paper published at
  {P}roceedings of the {$39$}th {ACM} {S}ymposium on {T}heory of {C}omputing
  (2007), 478--487.

\bibitem{perlis1}
R.~Perlis, \emph{On the equation $\zeta_k(s) = \zeta_{k'}(s)$}, Journal of
  Number Theory \textbf{9} (1977), no.~3, 342--360.

\bibitem{serreRepr}
J.-P. Serre, \emph{Linear representations of finite groups}, Springer New York,
  NY, 1977.

\bibitem{siegel}
C.L. Siegel, \emph{Lectures on the geometry of numbers}, Springer Berlin,
  Heidelberg, 1989.

\bibitem{solomatin}
P.~Solomatin, \emph{Global fields and their {$L$}-functions}, Ph.{D}. thesis,
  Leiden University, 2021.

\bibitem{sutherland}
A.~Sutherland, \emph{Arithmetic equivalence and isospectrality}, MIT Lecture
  Notes of Mini-Course in Topics in Algebra (2018).

\bibitem{sutherland_strong}
\bysame, \emph{Stronger arithmetic equivalence}, Discrete analysis \textbf{23}
  (2021), 23p.

\end{thebibliography}

\end{document}